\documentclass[12pt]{article}

\usepackage{lmodern,amsmath,amssymb,amsthm,mathrsfs,graphicx,float,xcolor,mathtools,comment,tikz,circuitikz}
\usepackage{hyperref}

\usepackage[a4paper, margin=25mm]{geometry}

\DeclarePairedDelimiter{\n}{\lVert}{\rVert}
\newcommand{\floor}[1]{\left\lfloor #1 \right\rfloor}
\newcommand{\ceil}[1]{\left\lceil #1 \right\rceil}

\newtheorem{thm}{Theorem}[section]
\newtheorem{prop}[thm]{Proposition}
\newtheorem{rema}[thm]{Remark}
\newtheorem{cor}[thm]{Corollary}
\newtheorem{lemma}[thm]{Lemma}
\newtheorem{defn}[thm]{Definition}
\newtheorem{prob}[thm]{Problem}

\newcommand{\dd}[1]{\,\mathrm{d}#1}
\def\bz{\mathbf{z}}
\def\bw{\mathbf{w}}
\def\diam{\mathrm{diam}}
\def\Loss{\mathcal{L}}
\def\P{\mathbb{P}}
\def\R{\mathbb{R}}
\def\Z{\mathbb{Z}}

\DeclareMathOperator{\hull}{hull}
\def\l{\lambda}
\def\d{\delta}
\def\eps{\varepsilon}

\title{The Loss of Tension in an Infinite Membrane with Holes of Decaying Spatial Density}

\author{Lukas Early\footnote{University of North Carolina at Chapel Hill, Department of Statistics and Operations Research, USA. Most of the work was however done when Lukas Early was a student at Lund University.}  \and Stanislav Volkov\footnote{Lund University,  Centre for Mathematical Sciences, Sweden
} }

\begin{document}
\maketitle
\begin{abstract}
What is the effect of randomly removing material from an infinite stretched membrane? Under what conditions can the membrane still sustain tension? This problem was introduced by Robert Connelly in connection with applications of rigidity theory (see e.g.~\cite{Con}) in the natural sciences, and was later studied in~\cite{MRV}; a discrete version was also considered in~\cite{CRV}.

We study a mathematical framework based on a non-homogeneous Poisson point process whose intensity $\lambda$ tends to zero at infinity. The hole shapes are i.i.d.\ and independent of their locations. We show that if the intensity does not decay too quickly, then tension is still lost throughout the whole plane, as in the homogeneous model studied in~\cite{MRV}. Conversely, we give sufficient conditions under which complete loss of tension does not occur. Thus both destruction and non-destruction regimes are possible even when the intensity tends to zero, indicating a phase transition in the model.

The processes studied here are closely related to bootstrap percolation.
\end{abstract}
\smallskip


\section{Introduction}\label{introMRV}
Consider an infinite stretched two-dimensional membrane, and let us define the notion of ``holes'' in this membrane.
\begin{defn}\label{def:hole}
{\em A hole} $H$ is a compact convex subset of $\R^2$ with non-empty interior,  with a distinguished interior point called its {\em centre}. The {\em diameter} of a hole is the quantity $\diam(H)=\sup_{x,y\in H} \n{x-y}$, where here and later throughout the paper $\n{z}$ denotes the Euclidean distance between $z\in\R^2$ and the origin ${\bf 0}=(0,0)$.
\end{defn}
A number of holes are punched in this membrane. Removing such a hole from the membrane leads to a redistribution of tension in the remaining region. With multiple holes, the behaviour becomes more complex: if two holes overlap, tension is lost at least on the convex hull of their union, not just their union itself.  If this convex hull intersects another hole, tension is now also lost on the convex hull of the union of the previous region and the new hole. Iterating this procedure produces the propagation of loss of tension schematically illustrated in Figure~\ref{fig123}. Note that the convex hull is a closure operator: it is extensive, monotone, and idempotent, and as a result, the order in which the convex hulls are taken does not matter. Consequently, we can define the limiting loss-of-tension region
(i.e., the infinite union of the convex hulls at each stage).

\begin{figure}
    \centering
\includegraphics[width=0.18\linewidth]{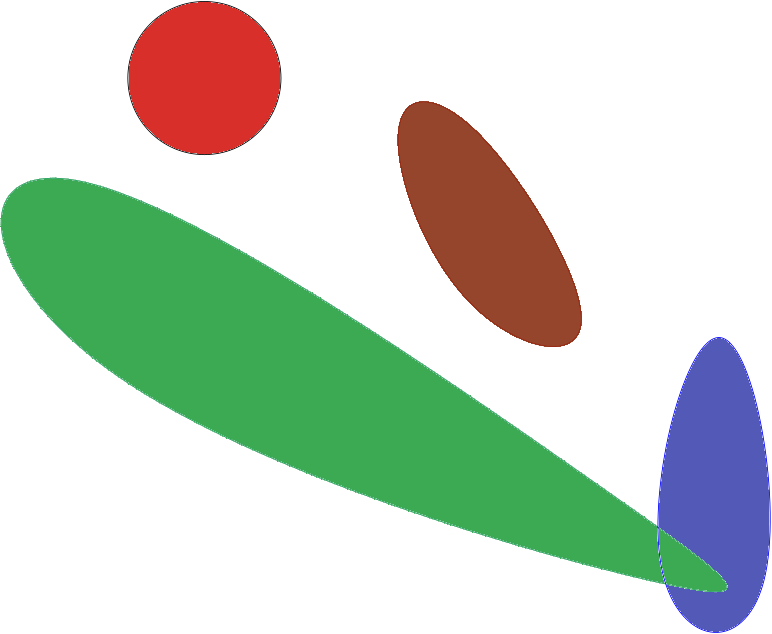}
\includegraphics[width=0.05\linewidth]{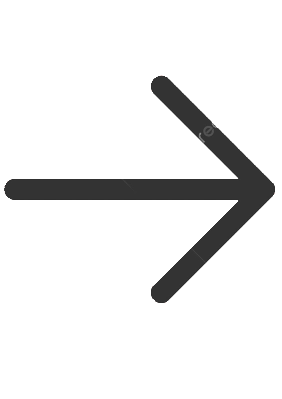}
\includegraphics[width=0.18\linewidth]{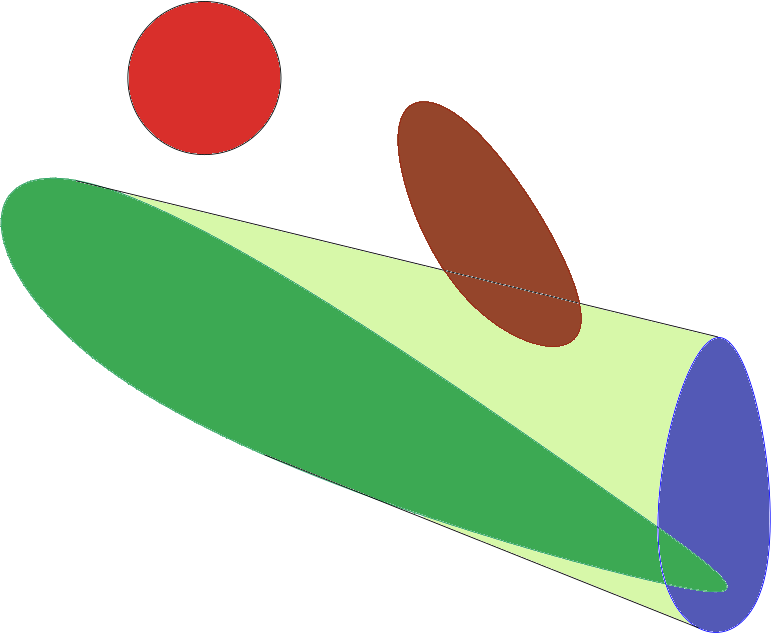}
\includegraphics[width=0.05\linewidth]{arrow.png}
\includegraphics[width=0.18\linewidth]{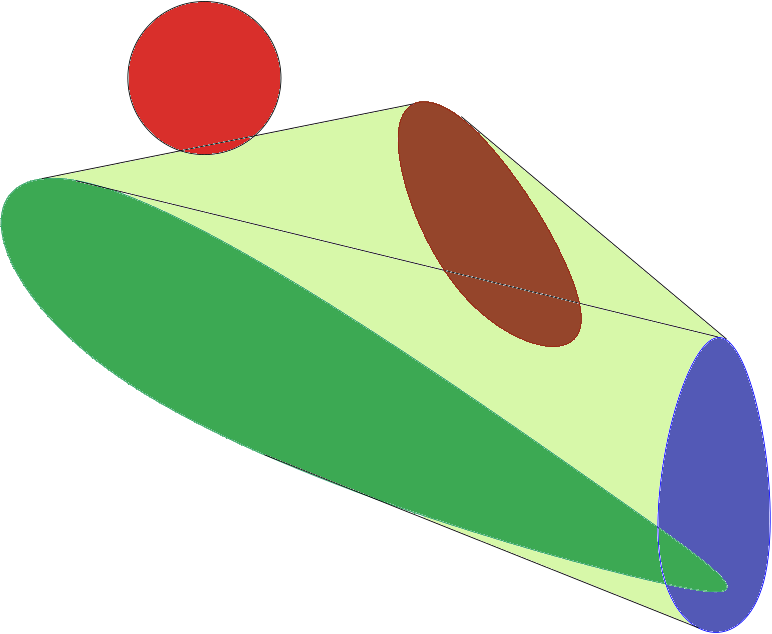}
\includegraphics[width=0.05\linewidth]{arrow.png}
\includegraphics[width=0.18\linewidth]{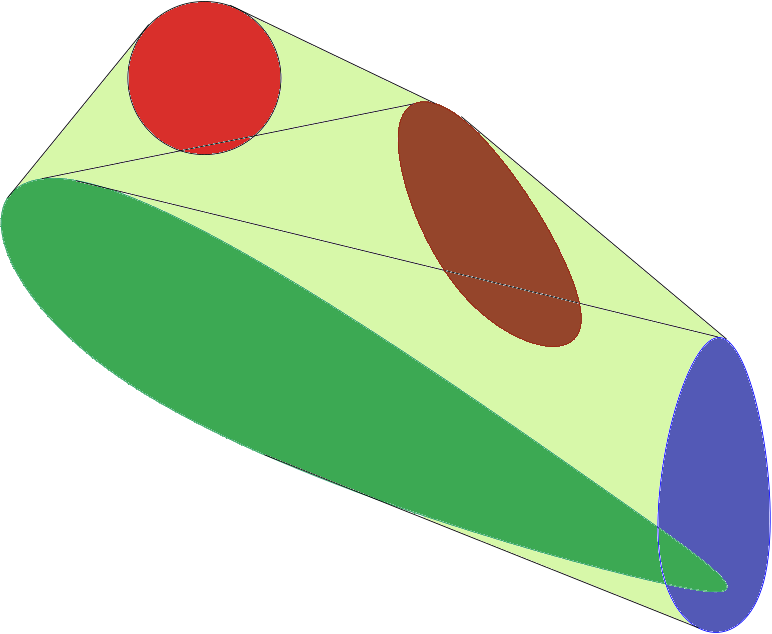}
    \caption{Sequential growth of convex hulls}
    \label{fig123}
\end{figure}

Formally, for $A\subseteq\R^2$, let $\hull(A)$ denote the convex hull of $A$,  the smallest convex set containing $A$, i.e.,
$$
\hull(A)=\bigcap_{C\subseteq \R^2\text{ is convex, } C\supseteq A} C.
$$
\begin{defn}
For a set \(A\subseteq \mathbb R^2\), let \(\mathcal C(A)\) be the collection of connected components of~\(A\). The {\em propagation of loss of tension} is the operator on the subsets of the plane $\phi:\mathscr{P}(\R^2)\to\mathscr{P}(\R^2)$ such that
\[
\phi(A)=\bigcup_{C\in\mathcal C(A)} \hull(C).
\]
\end{defn}
The process of propagation of loss of tension can be iterated. Let $\Loss_0=A$ and for $k\ge 1$  define $\Loss_k=\phi(\Loss_{k-1})$. Since $\{\Loss_k\}_k$ is an increasing sequence of sets, we can define the limit
$$
\Loss_\infty=\bigcup_{k=0}^\infty \Loss_k.
$$
If $\Loss_\infty=\R^2$, we say that {\em the whole membrane is destroyed} by the loss of tension propagation.
 
It was shown in~\cite{MRV} that if the initial configuration $\Loss_0$ consists of a collection of holes whose shapes are chosen randomly from the same distribution, and their centres are distributed according to a {\em homogeneous} spatial Poisson Point Process (hereafter abbreviated as {\rm PPP}) with constant intensity $\lambda>0$, then the loss of tension process above will eventually cover all of $\R^2$, leading to loss of tension everywhere, i.e., $\Loss_\infty=\R^2$. Notably, this fact holds for any value of $\lambda$, however small it is, implying that there is no phase transition in this model. This result is similar to those in bootstrap percolation models, which exhibit a similar behaviour, see e.g.~\cite{AL}; for a three-dimensional version see~\cite{vanEnter}.

The proof of the above fact heavily depends on the fact that the placement of holes by the PPP is homogeneous, so by e.g.\ the Borel-Cantelli lemma, there must be a region {\it somewhere} in the plane with a high concentration of holes, from which, using the sequential process of loss of tension, the plane gets covered. This method can be easily extended to the case when the intensity $\l=\l(\bz)$ depends on $\bz\in\R^2$ but $\liminf_{\n{\bz}\to\infty} \l(\bz)>0$. 
It is therefore natural to ask whether the complete loss of tension can occur when the intensity of the Poisson process decreases to zero. In this paper, we show that there are two functions $\bar\l,\underline{\l}:\R_+\to \R_+$, both converging to zero, such that the whole plane will be covered by the process if $\l(\bz)\ge \bar\l(\n{\bz})$, and, conversely, that this will not happen a.s.\ when $\l(\bz)\le \underline{\l}(\n{\bz})$, thus suggesting a phase transition in the model.

The model described in this paper is also related to entanglement percolation~\cite{Madras}.

\section{Main results}
Assume that centres of holes are placed at some countable subset $X\subset \R^2$; the hole centred at $\bz\in X$ is denoted by $H_{\bz}$. Then the initial configuration of holes is given by
$$
\Loss_0=\Loss_0(X)=\bigcup_{\bz\in X} H_{\bz}\subseteq \R^2.
$$

We start with the simple case where all the holes are unit discs and the hole centres are the natural centres of those discs. Let $\Loss_\infty$ be the limiting loss-of-tension region. Note that $\Loss_\infty$ depends on $X$ only.
\begin{thm}\label{th1}
Suppose that all the holes are unit discs, with their natural centres being the holes' centres. Let $X$ be a PPP with the intensity $\l(\bz)$, such that for large~$\n{\bz}$ we have $\l(\bz)\ge g(\n{\bz})$ where
\begin{equation}\label{gdef}
g(r)= \frac{1}{\log^{1/2-\d} r}\qquad r>1    
\end{equation}
for some $\delta\in\left(0,\frac12\right)$. Then the membrane will eventually be destroyed a.s.,  i.e.\ $\P(\Loss_\infty(X)=\R^2)=1$.
\end{thm}

Let 
$$B_r(\bz_0)=\{\bz\in\mathbb R^2:\n{\bz-\bz_0}\le r\}
$$
 denote a disc of radius $r>0$ centred at $\bz_0\in \R^2$.
\begin{defn}
Let $n_0$ be a sufficiently large positive integer. For each $n=n_0+1,n_0+2,\dots$  define {\em the $n$th  seed} as $B_{R_n}(\bw_n)$ where 
$$
\bw_n=(\ell_n,0),\qquad \ell_n=n\log^2 n,\qquad  R_n=\log^{1-\d} n.
$$
\end{defn}
\begin{lemma}\label{lemind}
Assuming $n_0$ is sufficiently large, the seeds do not intersect; moreover, the distance between consecutive seeds, say the $n$th and the $(n+1)$st, grows much faster than $R_n^2$ as $n\to\infty$.
\end{lemma}
\begin{proof}
The distance between the centres of the consecutive seeds (which are the closest to each other) satisfies
\begin{align*}
 (n+1)\log^2(n+1)-  n\log^2 n \ge \log^2 (n+1)\gg  2 \log^{2-2\d} (n+1) =2 R_{n+1}^2
\end{align*}
for sufficiently large $n$, since $\delta>0$. 
Hence $B_{R_n}(\bw_n)\cap B_{R_{n+1}}(\bw_{n+1})=\emptyset$ and 
$$
\inf\left\{\n{\bz'-\bz''}:\ 
\bz'\in B_{R_n}(\bw_n),\ \bz''\in B_{R_{n+1}}(\bw_{n+1})\right\}
\to \infty\quad\text{as }n\to\infty.
$$
Since \(\log^2 n / R_n^2=\log^{2\delta}n\to\infty\), the distance
between consecutive seed centres is much larger than \(R_n^2\), and in particular much larger than \(R_n+R_{n+1}\). Hence the seeds are disjoint for all sufficiently large \(n\).
\end{proof}

\begin{defn}
We call the $n$th seed {\em good} if $B_{R_n}(\bw_n)\subseteq\Loss_1$ regardless of the configuration {\em outside} of the seed itself.
\end{defn}

From now on, let $X$ be a realization of the PPP with intensity $\l$ as in Theorem~\ref{th1} and let $\Loss_0$ be a collection of holes centred at the points of $X$, formally
$$
\Loss_0=\{\bz\in \R^2:\ \n{\bz-\bz'}\le 1\text{ for some }\bz'\in X\}.
$$

\begin{figure}
    \centering
    \includegraphics[width=0.3\linewidth]{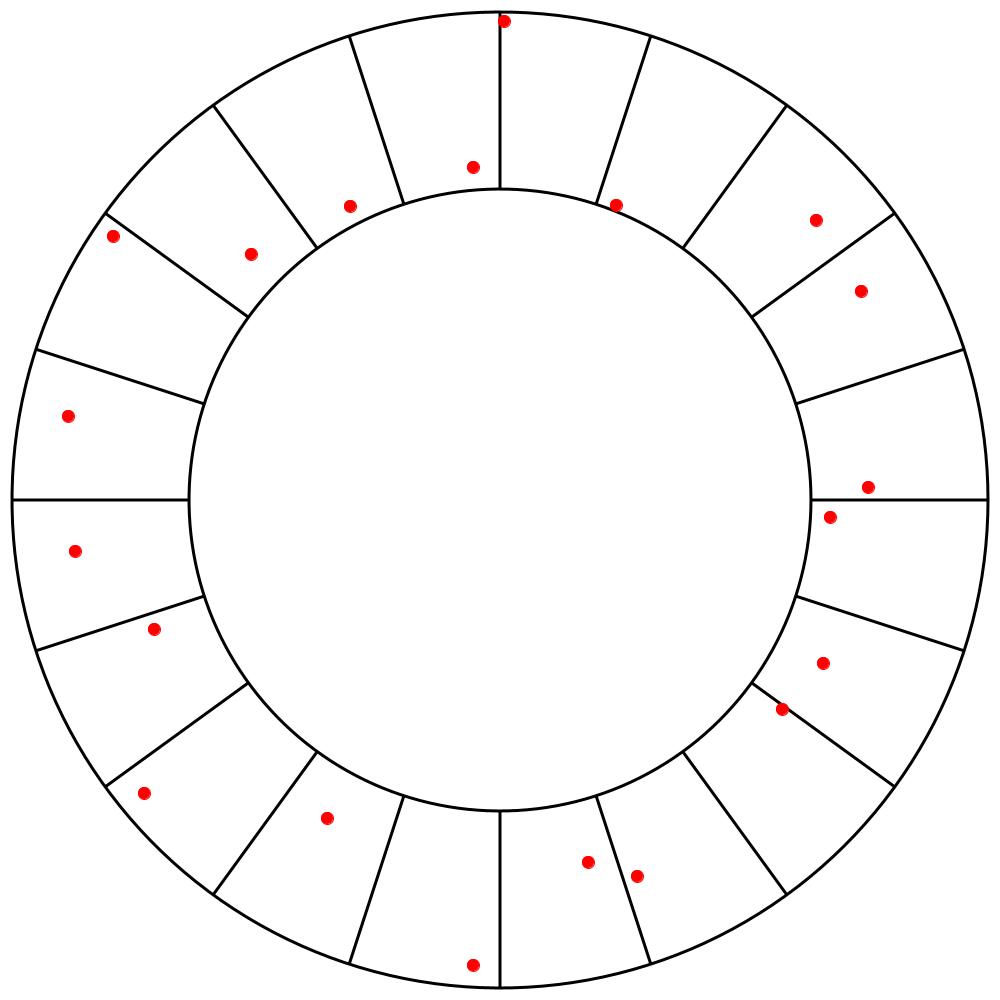}
    \qquad
    \includegraphics[width=0.4\linewidth]{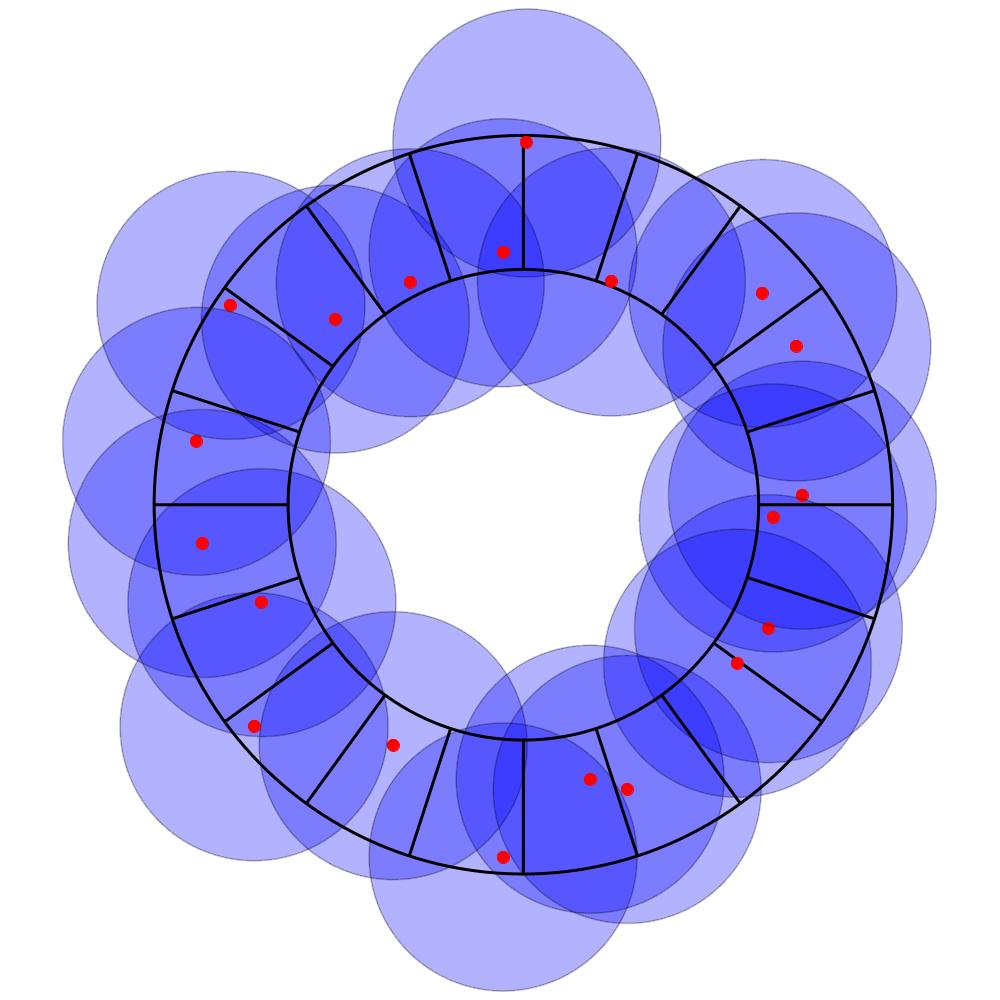}
    \caption{A good seed: each annular sector contains a point of $X$, and the convex hull of the corresponding discs completely covers the seed.}
    \label{fig:1}
\end{figure}

\begin{lemma}\label{lem1}
With probability $1$, there are infinitely many good seeds. \end{lemma}
\begin{proof}
Let $A_n$ be the event ``the $n$th seed is good''.  Since the seeds are disjoint and \(A_n\) depends only on the configuration of \(X\) inside \(B_{R_n}(\mathbf w_n)\), the events \(A_n\) are independent by the independence of Poisson processes on disjoint regions. We will show that $\P(A_n)\ge 1/n$, and then the result will follow from the second Borel-Cantelli lemma and the fact that the harmonic series diverges.

Indeed, split the ring 
$$
B_{R_n}(\bw_n)\setminus B_{R_n-1/2}(\bw_n)
=\{\bz\in \R^2:\ R_n-1/2< \n{\bz-\bw_n}\le R_n\}
$$
into $K_n:=\floor{13 R_n}$ equal annular sectors; each sector has an angle $\alpha_n=\frac{2\pi}{K_n}$. Then for large enough $n$ we have the following:
\begin{itemize}
\item[(a)] the maximum distance between two points of the same annular sector is smaller than
$$
\alpha_n\, R_n+\frac12=\frac{2\pi R_n}{\floor{13 R_n}}+\frac12<1;
$$
\item[(b)] the maximum distance between two points of adjacent annular sectors is smaller than
$$
2\alpha_n\, R_n+\frac12=\frac{4\pi R_n}{\floor{13 R_n}}+\frac12<2;
$$
\item[(c)] the area of each annular sector is
$$
\frac12 \left(R_n^2-\left(R_n-\frac12\right)^2\right)\alpha_n=\frac{\pi(R_n-1/4)}{\floor{13 R_n}}>\frac15.
$$
\end{itemize}
Now note that a sufficient condition for the $n$th seed to be good is that each annular sector contains at least one point of $X$, as then we will have a sequence of $K_n$ overlapping discs, whose convex hull covers $B_{R_n}(\bw_n)$ entirely.
Indeed, because of (a), each annular sector will be completely covered by the hole centred in it, implying that the piece of the circumference corresponding to this sector is covered. Because of (b), every two holes in the neighbouring sectors overlap.
See Figure~\ref{fig:1}.

Since the area of each annular sector is at least $\frac15$ by (c), and $\l(\bz)\ge g(\ell_n+R_n)$ for all $\bz\in B_{R_n}(\bw_n)$ as $(\ell_n+R_n,0)$ is the furthest point from the origin belonging to the $n$th seed, the number of points of $X$ in each of the annular sectors has a Poisson distribution with mean at least $\frac15 g(\ell_n+R_n)$, yielding
$$
\P(A_n)\ge \left[1-\exp\left(- \frac{g(\ell_n+R_n)}5\right)\right]^{K_n}
\ge \left(\frac{g(\ell_n+R_n)}{10}\right)^{K_n}=e^{K_n\log(g(\ell_n+R_n)/10)}
$$
where we used the fact that $1-e^{-x}\ge x/2$ for $0\le x\le 1$, which is applicable since $g(\ell_n+R_n)\to 0$.
Now
\begin{align*}
-K_n \log(g(\ell_n+R_n)/10)&=\floor{13 R_n} \log\left(10\,\log^{1/2-\d}(\ell_n+R_n)\right)  \\ & =(6.5-13\d+o(1)) \left(\log^{1-\d}  n\right) \left(\log \log n\right) 
<\log n,
\end{align*}
hence $\P(A_n)>e^{-\log n}=\frac{1}{n}$ and thus the lemma follows.
\end{proof}

\vspace{1cm}

Let $\bz_0\in \R^2$,
$$
M_{u}=\floor{10\sqrt{u}}, \quad 
\alpha_u=\frac{2\pi}{M_{u}}=\frac{\pi}{5\sqrt{u}}(1+o(1)) \text{ as }u\to\infty,
$$
and split the ring
\begin{align}\label{eq:defQ}
Q_u(\bz_0)=B_{u}(\bz_0)\setminus B_{u-\frac12}(\bz_0)
=\left\{\bz\in \R^2:\ u-\frac12< \n{\bz-\bz_0}\le u\right\}
\end{align}
centred at $\bz_0$ into $M_{u}$ equal annular sectors with the angle $\alpha_{u}$.

\begin{defn}
For a point $\bz_0=(x_0,y_0)$ and $\nu\in(0,1)$,  define the {\em forbidden cone} as
$$
F_\nu(\bz_0)=\{(x,y)\in\R^2:\ |y-y_0|<|x-x_0|^{\frac{1-\nu}{2}}\}.
$$
\end{defn}

\begin{figure}[ht]
    \centering
    \includegraphics[width=0.5\linewidth]{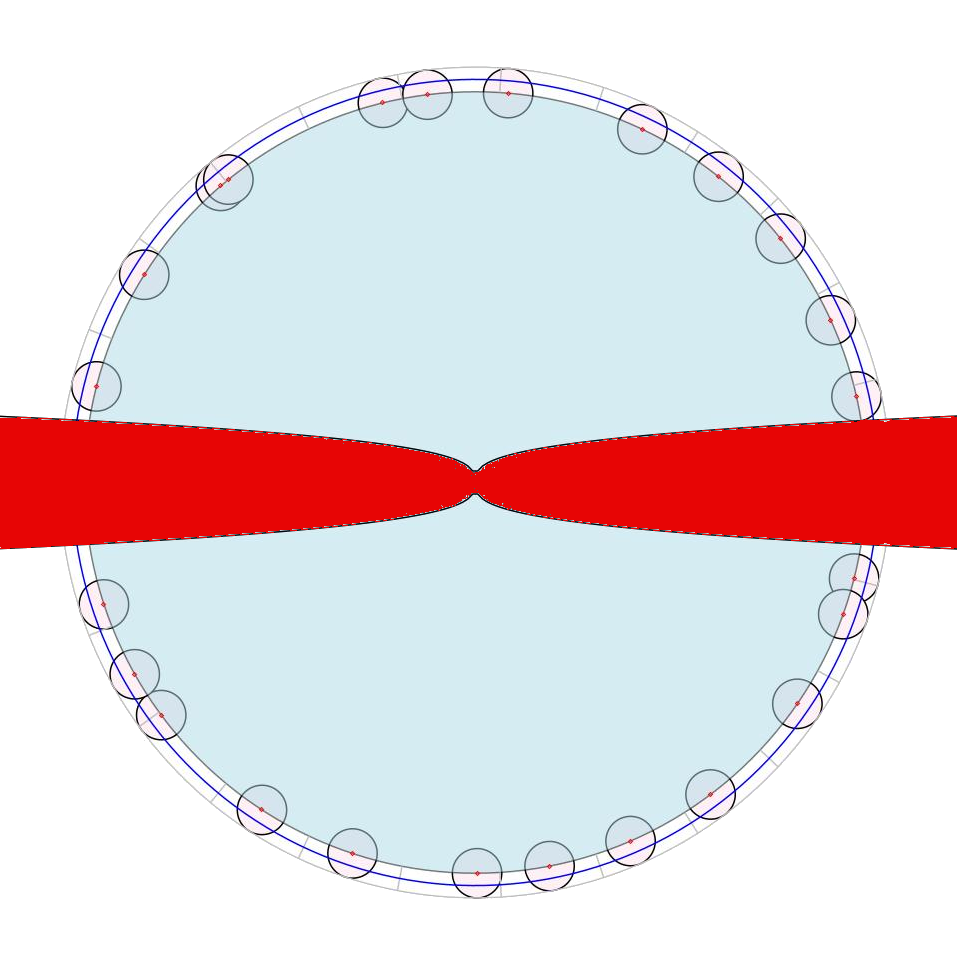}
    \caption{Forbidden cone shown in red.}
    \label{fig:forb}
\end{figure}

\begin{lemma}\label{lem2}
Fix $\nu\in(0,1)$. Suppose that  $u$ is sufficiently large and that $B_{u-1}(\bz_0)\subseteq \Loss_{k}$ regardless of the configuration outside of this disc $B_{u-1}(\bz_0)$.  Then, if each of the $M_u$ annular sectors of $Q_u(\bz_0)$ defined by~\eqref{eq:defQ} contains a point of $X$ outside the forbidden cone $F_\nu(\bz_0)$, then $B_u(\bz_0)\subseteq \Loss_{k+1}$.
 (See Figure~\ref{fig:forb}.)
\end{lemma}

\begin{proof}
We use a similar ``rim'' construction as the one used in the proof of Lemma~\ref{lem1}; w.l.o.g.\ $\bz_0=\mathbf{0}$.

Let us denote the consecutive points in the annular sectors as $\bz_1,\bz_2,\dots,\bz_{M_u}$. Then the point farthest from $\mathbf{0}$ of the hole $H_i$ centred at $\bz_i$ is $\tilde\bz_i:=\frac{\n{\bz_i}+1}{\n{\bz_i}}\bz_i$; denote its polar coordinates by $(\rho_i,\beta_i)$, $i=1,2,\dots,M_u$. It is enough to show that every side of the polygon with vertices \(\tilde\bz_1,\ldots,\tilde\bz_{M_u}\), ordered by angle, lies outside \(B_u(\mathbf 0)\). Since the vertices wind once around the origin, this implies that \(B_u(\mathbf 0)\) is contained in the polygon, and hence in the convex hull of the holes. Since each of the holes $H_1,H_2,\dots,H_{M_u}$ also intersects $B_{u-1}(\mathbf{0})$, this would mean that 
$$
\phi\left(B_{u-1}(\mathbf{0})
\cup H_{1}\cup H_{2}\cup\dots\cup H_{M_u}\right)\supseteq B_u(\mathbf{0}).
$$
This implies $B_{u}(\bz_0)\subseteq \Loss_{k+1}$.

Now it remains to show that each segment connecting two consecutive 
$\tilde\bz_i$ lies completely outside of $B_u(\mathbf{0})$. Because of rotational symmetry, it suffices to show this for $\tilde\bz_1$ and $\tilde\bz_2$ only. Indeed, we have
$$
\rho_1,\rho_2\ge u+\frac12,\quad |\beta_1-\beta_2|\le 2\alpha_{u},
$$
so the endpoints of the segment are definitely outside of $B_u(\mathbf{0})$. Let 
$$
\bar\rho=\bar\rho(\tilde\bz_1,\tilde\bz_2)=\min_{\gamma\in[0,1]} \n{\gamma \tilde\bz_1+(1-\gamma) \tilde\bz_2}
$$ 
be the shortest distance between the points of this segment and the origin; we will show that $\bar\rho\ge u$. Indeed, due to simple geometric arguments, for fixed $\beta_1-\beta_2$, the value $\bar\rho$ increases in both $\rho_1$ and $\rho_2$, hence its smallest value is achieved when $\rho_1=\rho_2=u+1/2$. As a result, $\bar\rho$ cannot be smaller than the height of the isosceles triangle with sides $u+1/2$ and the angle $|\beta_1-\beta_2|$, that is
\begin{align*}
\bar\rho&\ge  \left(u+\frac12\right)\cos\frac{\beta_1-\beta_2}{2}
\ge \left(u+\frac12\right)\cos\alpha_u
\ge \left(u+\frac12\right)\left(1-\frac{\alpha_{u}^2}{2}\right)
\\ &
=u+\frac12-\frac{\pi^2}{50}+o(1)>u,
\end{align*}
for large enough $u$, as required. 
See Figure~\ref{fig:prop}.
\begin{figure}[h]
    \centering
    \includegraphics[width=0.6\linewidth]{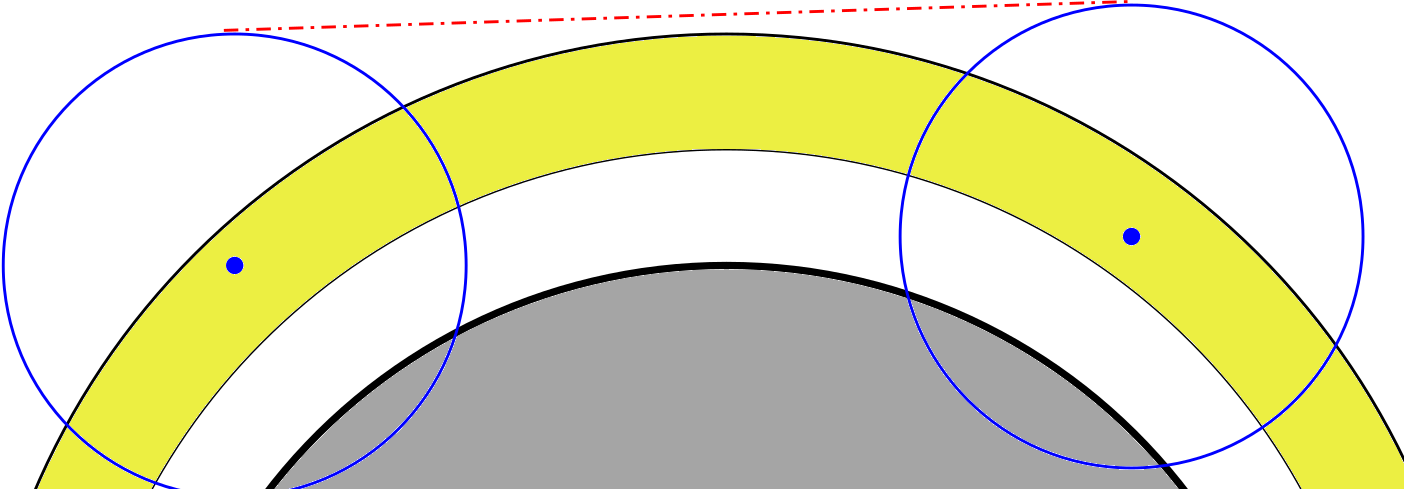}
    \caption{Propagation of convex hulls. There are two centres (blue dots) of the blue circles in the yellow ring $Q_u$; the red dashed line shows part of the border of the convex hull obtained in the next iteration, as both circles intersect the already destroyed part of the plane (inside of the black circle, a part of which is shown).}
    \label{fig:prop}
\end{figure}
\end{proof}

\begin{lemma}\label{lem3}
Fix $\nu\in(0,1)$. There exists a sequence $\{\gamma_n\}$, possibly depending on $\nu$, with $\gamma_n\to0$, such that, conditional on the $n$th seed being good, the following holds with probability at least $1-\gamma_n$: starting from $B_{R_n}(\bw_n)$ already destroyed, and using only holes whose centres lie outside $F_\nu(\bw_n)$, the propagation covers all of $\mathbb R^2$.
\end{lemma}

\begin{proof}
Define
\begin{align}\label{eq:lR}
u_j=\ell_j+R_j=j\log^2 j +\log^{1-\d}j,\qquad j=1,2,\dots.
\end{align}
Choose \(r_0\) such that \(\lambda(\mathbf z)\ge g(\|\mathbf z\|)\) whenever \(\|\mathbf z\|\ge r_0\). For sufficiently large \(n\),
\[
(\ell_n-r_0)^{(1-\nu)/2}>r_0,
\qquad B_{r_0}(\mathbf0)\subset F_\nu(\mathbf w_n).
\]
Consequently, for every 
$\mathbf z\in Q_{R_n+k}(\mathbf w_n)\setminus F_\nu(\mathbf w_n)$
we have
\[
\quad\lambda(\mathbf z)\ge g(\|\mathbf z\|)\ge g(u_n+k),
\text{ uniformly in \(k\ge1\).}
\]
W.l.o.g.\ assume that $n_0$ is such that $\ell_{n_0}>r_0$ and $(\ell_{n_0}-r_0)^{(1-\nu)/2}>r_0$.  Fix an $n\ge n_0$. We will construct a sequence of {\em independent} events $A_{n,k}$, $k=1,2,\dots$, such that 
$$
G_n\cap\bigcap_{i=1}^k A_{n,i}\subseteq \{B_{R_n+k}(\bw_n)\subseteq \Loss_{k+1}\}\qquad k\ge 1,
$$
where $G_n=\{$the $n$th seed is good$\}$, and, moreover, $A_{n,k}$ are independent of $G_n$. As a result, under this conditioning,
\begin{align}\label{eqBn}
\P(B_{R_n+k}(\bw_n)\subseteq \Loss_{k+1}\mid G_n)\ge \prod_{i=1}^k \P(A_{n,i}).    
\end{align}
Now it only remains to show that $1-\gamma_n:=\prod_{i=1}^\infty \P(A_{n,i})$  tends to $1$ as $n\to\infty$.

Let~$A_{n,k}$ be the event that each of $M_{R_n+k}$ annular sectors of $Q_{R_n+k}(\bw_n)\setminus F_\nu(\bw_n)$ has a point of $X$; since these sectors are not overlapping, these events are independent. Together with Lemma~\ref{lem2}, this implies that $A_{n,k}$s have exactly the required properties described above. For large $n$ each annular sector has the area 
$$
\frac{\pi\left(\left(R_n+k\right)^2-\left(R_n+k-1/2\right)^2\right)}{M_{R_n+k}}=\frac{\pi\left(R_n+k-\frac14\right)}{\floor{10\sqrt{R_n+k}}}>0.3\sqrt{R_n+k}
$$
which holds even for the sectors which intersect the forbidden zone, as the area of $F_\nu(\bw_n)\bigcap Q_{R_n+k}(\bw_n)$ is of order $(R_n+k)^{\frac{1-\nu}2}=o(\sqrt{R_n+k})$. 

Since, after removing the forbidden cone, each such sector has area at least $0.3\sqrt{R_n+k}$, the number of points of $X$ in it is Poisson with mean at least $0.3\sqrt{R_n+k}\,g(u_n+k)$.  Since \(g\) is decreasing for large arguments, and since \(n\) is chosen large enough so that the lower bound \(\lambda(\mathbf z)\ge g(\n{\mathbf z})\) applies on the relevant region, we get
\[
\lambda(\mathbf z)\ge g(\n{\mathbf z})\ge g(u_n+k).
\]
As a result, we have
$$
\P(A_{n,k})\ge \left(1-e^{-0.3\sqrt{R_n+k}\,g(u_n+k)}\right)^{M_{R_n+k}}
\ge \left(1-e^{-0.3\sqrt{R_n+k}\,g(u_n+k)}\right)^{10\sqrt{R_n+k}}
$$
that is,
$$
\log \P(A_{n,k})\ge 10 \sqrt{R_n+k}\ \log\left(1-e^{-0.3\,a(n,k)}\right)
\qquad\text{where }
a(n,k)= \frac{\sqrt{R_n+k}}{\log^{1/2-\d}
(u_n+k)}.
$$
Now, 
$$
\frac{\partial a(n,k)}{\partial k}=
\frac{k\left[\log (u_n+k)-1+2\d \right]+u_n\log(u_n+k)-(1-2\d) R_n}
{2(u_n+k)\ \log^{3/2-\d}(u_n+k)\ \sqrt{k+R_n} }>0
$$
(as $u_n>R_n=\log^{1-\d}n\gg 1$), so that $a(n,k)$ is increasing in $k\ge 0$ and thus
$$
a(n,k)\ge a(n,0)=\frac{\log^{\frac{1-\d}2}n}{\log^{\frac12-\d}(n\log^2 n+\log^{1-\d} n)}
=(1-o(1))\log^{\frac{\d}2}n\to\infty\qquad\text{as }n\to\infty.
$$
As a result, for sufficiently large $n$   we have $e^{-0.3a(n,k)}<3/4$ for all $k\ge 0$. Since 
$$
\log(1-x)\ge -2x\text{  for }0\le x<3/4,
$$
we conclude that
$$
\log \P(A_{n,k})\ge- 20\, e^{-0.3\,a(n,k)} \, \sqrt{R_n+k},\qquad k\ge 1.
$$
and thus
\begin{align}\label{eqSn}
\prod_{k=1}^\infty \P(A_{n,k})\ge e^{-20\, S_n}
\qquad\text{where }
S_n=\sum_{k=1}^\infty e^{-0.3\,a(n,k)} \, \sqrt{R_n+k}.
\end{align}
Split the sum into three parts, $S_n=\mathrm{(I)}+\mathrm{(II)}+\mathrm{(III)}$. First, for $k\le \log n$ we have
$$
\mathrm{(I)}:=\sum_{k=1}^{\floor{\log n}} e^{-0.3\,a(n,k)} \, \sqrt{R_n+k}
\le  e^{-0.3\ \log^{{\d}/2} n }\ \sqrt{2\log n}\cdot \log n
< e^{-0.3\ \log^{{\d}/2} n }\cdot  \log^2 n.
$$
Secondly, for $\log n<k\le n \log^2 n$ 
\[
a(n,k)\ge \frac23\,\frac{\sqrt{k}}{\log^{1/2-\delta}n}
=\frac{2\sqrt{k}}3\,g(n)
\]
for all large \(n\), therefore
\[
\mathrm{(II)}:=\sum_{k=\lceil\log n\rceil}^{\lfloor n\log^2 n\rfloor}
e^{-0.3\,a(n,k)} \, \sqrt{R_n+k}
\le
\sum_{k=\lceil\log n\rceil}^{\lfloor n\log^2 n\rfloor}
2\sqrt{k}\exp\left\{-0.2\,g(n)\sqrt{k}\right\}.
\]
Since the function $f(x)=\sqrt{x}e^{-0.2g(n)\sqrt{x}}$ is decreasing in $x$ as $\sqrt{x}>5/g(n)$ and $x\ge \log n$, we have
\begin{align*}
\mathrm{(II)}&\le 2\int_{\log n-1}^{n\log^2 n} \sqrt{x}e^{-0.2g(n)\sqrt{x}}\dd{x}
\le 
2\int_{\frac14\log n}^{\infty} \sqrt{x}e^{-0.2g(n)\sqrt{x}}\dd{x}
\\
&=(5+o(1))\, e^{-0.1\log^{\d} n}\,\log^{3/2-\d} n.
\end{align*}
Finally, for $k>n\log^2 n$ we have
$$
a(n,k)\ge \frac{\sqrt{k}}{\log^{1/2-\d}k}(1+o(1))\ge \frac{k^{1/4}}{0.3}
$$
yielding
\begin{align*}
\mathrm{(III)}&:=\sum_{k=\ceil{n\log^2 n}}^{\infty} e^{-0.3\,a(n,k)} \, \sqrt{R_n+k}
\le  
2 \sum_{k=\ceil{n\log^2 n}}^{\infty} \sqrt{k}\, e^{-k^{1/4}}
\\ &
<2\int_{n\log^2 n-1}^\infty \sqrt{x}\ e^{-x^{1/4}}\dd{x}
\le 2\int_{n}^\infty \sqrt{x}\ e^{-x^{1/4}}\dd{x}=(8+o(1)) n^{5/4}
\, e^{-n^{1/4}}
\end{align*}
as the function $\sqrt{x}\ e^{-x^{1/4}}$ is decreasing monotonically for large $x>0$.

Since $\lim_{x\to\infty} x^{c_1} \exp\left(-c_2 x^{c_3}\right)=0$ for any choice of $c_1,c_2,c_3>0$, we see that $\mathrm{(I)}$, $\mathrm{(II)}$  and $\mathrm{(III)}$ all go to zero as $n\to\infty$. Therefore, $S_n\to 0$ as $n\to \infty$. Moreover,
$$
\P\left(\bigcap_{k\ge 1}A_{n,k}\right)
=\prod_{k\ge 1}\P(A_{n,k})\ge e^{-20S_n}.
$$
On $G_n\cap\bigcap_{k\ge1}A_{n,k}$ we have
$$
B_{R_n+k}(\bw_n)\subseteq\Loss_{k+1}\qquad\text{for every }k,
$$
and since $R_n+k\to\infty$, we conclude that $\bigcup_{k\ge 1}B_{R_n+k}(\bw_n)=\R^2$.
\end{proof}

\begin{proof}[Proof of Theorem~\ref{th1}]
Let $\nu\in(0,\d)$. Then for $n$ large enough, for each $k=1,2,\dots,n-1$ the $k$th seed lies entirely inside of the forbidden zone $F_\nu(\bw_{n})$. Indeed,   
\[
|y|\le R_k\text{  and }|x-\ell_n|\ge \ell_n-\ell_k-R_k
\quad \text{ when }\bz=(x,y)\in B_{R_k}(\bw_k).
\]
For $k<n$, the right-hand side of the second inequality is minimized  when $k=n-1$ (as $u_k=\ell_k+R_k$ is monotone in~$k$, see~\eqref{eq:lR}). Hence, for all sufficiently large $n$,
\[
|x-\ell_n|^{(1-\nu)/2} \ge
\left(\ell_n-\ell_{n-1}-R_{n-1}\right)^{(1-\nu)/2}
\gg \log^{1-\delta} n \ge R_k\ge |y|,
\]
(see the proof of Lemma~\ref{lemind}) because $\nu<\delta$. Therefore $B_{R_k}(\bw_k)\subseteq F_\nu(\bw_n)$.

Fix an $\eps>0$. Let $A_n$ be the event that the $n$th seed is good; note that $A_n$ depends solely on $X\bigcap B_{R_n}(\bw_n)$. 
Let \(C_n\) be the event that, starting with the whole disc
\(B_{R_n}(\bw_n)\) already destroyed, and using only holes with centres in
\[
X\setminus\left(B_{R_n}(\bw_n)\cup F_\nu(\bw_n)\right),
\]
the propagation {\em does not} cover all of \(\mathbb R^2\).
By Lemma~\ref{lem3} there is $n_1$ such that for all $n\ge n_1$
\[
\P(C_n)\le \eps.
\]
Let $\tau=\inf\{n\ge n_1:\ A_n\text{ occurs}\}$; by Lemma~\ref{lem1}
$\tau<\infty$ a.s. On the event \(\{\tau=n\}\cap C_n^c\),
the $n$th seed is good and the outside propagation succeeds; hence
\(\Loss_\infty(X)=\mathbb R^2\). Therefore,
\begin{align*}
\P(\Loss_\infty(X)\ne\mathbb R^2)
\le \P(C_\tau)
&=\sum_{n\ge n_1}\P(\tau =n\text{ and }C_n)
\\ &
=\sum_{n\ge n_1}\P(\tau=n)\P(C_n)
\le
\eps\sum_{n\ge n_1}\P(\tau=n)
\le \eps
\end{align*}
because \(C_n\) is independent of \(\{\tau=n\}\).
Since $\eps$ is arbitrary, the result follows.
\end{proof}

\begin{cor}\label{COR}
Suppose that the holes are chosen i.i.d.\ from some collection of holes $\mathcal{H}$ and for $\n{\bz}$ large enough the intensity of the PPP satisfies $\lambda(\bz)\ge g(\n{\bz})$, where $g$ is defined by~\eqref{gdef}. Then the membrane will eventually be destroyed a.s. 
\end{cor}
\begin{proof}

For each hole $H\in\mathcal{H}$ let $r_H$ be the maximum radius of the discs centred at the centre of the hole, completely contained inside of $H$. Then $r_H$ is a positive random variable, and hence there exists $\epsilon>0$ such that $\P(r_H\ge \epsilon)\ge 1/2$. Now, after placing the holes on the membrane, throw away all those with $r_H<\epsilon$; by this, we make the task of destroying the membrane only harder, because of the monotonicity of $\phi$. As a result, each of the remaining holes contains the disc of radius $\epsilon$, centred at the points of the PPP. Now we have reduced the model to the one in Theorem~\ref{th1}, except that the radius of all discs is $\epsilon$ and the intensity has been reduced by at most a factor \(2\).

The model can now be rescaled with the new intensity $\tilde\lambda(\bz)$ of the PPP  such that $\tilde\lambda(\bz)\ge \frac{\epsilon^2}2\lambda(\epsilon\, \bz)$, so that all discs have unit radius.  The multiplicative constant $\epsilon^2/2$ can be absorbed by replacing $\delta$ with a smaller positive value, and the conditions of Theorem~\ref{th1} are fulfilled.
\end{proof}

\section{Non-propagation}
Throughout this section, we will assume that almost surely any bounded Borel set  $A\subset \R^2$ has finitely many holes centred in it, which is equivalent to
\begin{align}\label{eq:noexplosion}
\Lambda(A):=\iint_A \l(\bz)\dd{\bz}<\infty,    
\end{align}
as the number of such centres has a Poisson distribution with parameter $\Lambda(A)$.

Obviously, the membrane will not be destroyed if there are only finitely many holes on the plane, regardless of the set $\mathcal{H}$, which leads to the following straightforward statement.
\begin{prop}\label{prop2}
Suppose
\begin{equation}\label{eqfinite}
\Lambda(\R^2)=\iint_{\R^2} \l(\bz)\dd{\bz}<\infty.    
\end{equation}
Then the membrane will not be destroyed a.s., that is, $\P\left(\Loss_\infty=\R^2\right)=0$.
\end{prop}
\begin{proof}
The total number of holes on the whole plane has a Poisson distribution with parameter $\Lambda(\R^2)<\infty$. Therefore, the number of holes is a.s.\ finite, and the union of all of them lies inside of a bounded set a.s.\ (see Definition~\ref{def:hole}). This implies the statement.
\end{proof}
Note that if $\l(\bz)=f(\n{\bz})$ then~\eqref{eqfinite} is equivalent to
$$
\int_0^\infty rf(r) \dd{r}<\infty
$$
which will be true e.g.\ if $f(r)\le \frac1{r^{2+\eps}}$ for some $\eps>0$.

We can strengthen the above proposition as follows.
\begin{thm}\label{th3}
Assume that for some constant $D>0$  all holes are bounded by $D$ in size, i.e.\ $\sup_{H\in\mathcal{H}} \diam(H)< D$, and the  Poisson process is such that for large $\n{\bz}$ its intensity satisfies $\l(\bz)\le f(\n{\bz})$ where
$$
\int_0^\infty f(r)\,\dd{r}<\infty 
$$
and the function $f(r)$ is decreasing for large $r$.
 Then the membrane will not be destroyed a.s. 
\end{thm}
For example, any $\lambda$ with the corresponding $f(r)=\frac{1}{r\log^\gamma r}$ for some $\gamma >1$ and sufficiently large $r$ will satisfy the condition of the theorem.
\begin{proof}[Proof of Theorem~\ref{th3}]
It is sufficient to show that we can find an integer $h$ such that the horizontal strip
$$
S_h=\{(x,y)\in\R^2:\ 2Dh< y \le 2D(h+1)\}
$$
is completely free from the points of the Poisson process; call this event $A_h$ (see Figure~\ref{fig:strip}.)
\begin{figure}[h]
    \centering
    \includegraphics[width=0.5\linewidth]{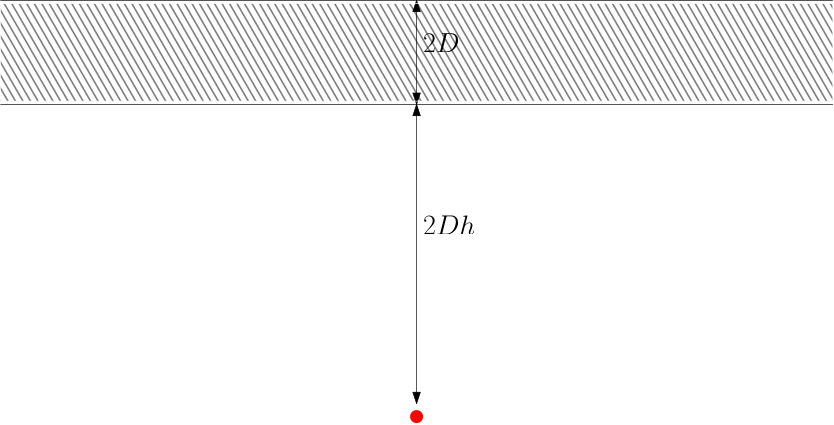}
    \caption{An impenetrable strip $S_h$ separating two parts of the plane}
    \label{fig:strip}
\end{figure}
Then the convex hull of all the holes below this strip and the convex hull of those above the strip will not have common points, and hence no point on the line 
$$\{(x,y)\in\R^2:\ y=2Dh+D\}$$ will be covered by $\Loss_\infty$.

Assume that $h$ is so large that $f(r)$ is decreasing for $r\ge 2Dh$. Now, the probability of $A_h$, i.e.,  that no point of the Poisson process lands in $S_h$ equals $\P(A_h)=e^{-\Lambda_h}$ where
$$
\Lambda_h=\iint_{S_h}\lambda(\bz) \dd{\bz}
\le \int_{-\infty}^{\infty} \left[\int_{2Dh}^{2D(h+1)} \lambda(x,y) \dd{y}\right]\dd{x}
\le 4D\int_{0}^{\infty} f\left(\sqrt{x^2+4D^2h^2}\right)\dd{x}
$$
for large enough $h$, and $\bz:=(x,y)$ throughout the proof. Set $c(h)=\frac{\log h}{5D f(2Dh)}$. The integral in the RHS is bounded by
$$
4D\int_{0}^{c(h)} f(2Dh) \dd{x}+ 4D\int_{c(h)}^{\infty}  f\left(x\right)\dd{x}=\frac45 \log h+o(1)<\log h \quad \text{ for large }h
$$
where $o(1)\to 0$ as $h\to+\infty$ since $c(h)\to\infty$.
Hence $\Lambda_h<\log h$ for large positive $h$ yielding that $\P(A_h)>\frac 1h$. Since the events $\{A_h\}$ are independent and $\sum_{h=1}^\infty \frac1h=\infty$, it implies that a.s.\ there exists an $h$ for which $A_h$ occurs and hence $\Loss_\infty\ne\R^2$.
\end{proof}

\begin{figure}[!ht]
\centering
\resizebox{0.5\textwidth}{!}{\begin{circuitikz}
\tikzstyle{every node}=[font=\fontsize{18.2pt}{23.7pt}\selectfont]
\draw  (1.25,15.125) rectangle (21.25,13.875);
\draw  (2.5,16.375) rectangle (3.75,0.125);
\draw  (1.25,2.625) rectangle (21.25,1.375);
\draw  (18.75,16.375) rectangle (20,0.125);
\draw  (6.25,11.375) circle (1.25cm);
\draw  (8.75,11.375) circle (1.25cm);
\draw  (7.5,10.125) circle (1.25cm);
\draw  (5.625,8.875) circle (1.25cm);
\draw  (6.25,6.375) circle (1.25cm);
\draw  (7.5,7.625) circle (1.125cm);
\draw  (7.5,4.5) circle (1.25cm);
\draw  (9.25,5.875) circle (1.125cm);
\draw  (9.375,8.875) circle (1.25cm);
\draw  (10.625,7) circle (1.25cm);
\draw  (11.75,9.375) circle (1.25cm);
\draw  (10.625,11) circle (1.125cm);
\draw  (13.125,6.875) circle (1.125cm);
\draw  (13.25,8.5) circle (1.375cm);
\draw  (11.875,5.625) circle (1.125cm);
\draw  (10.5,4.5) circle (1.125cm);
\draw  (13.5,11) circle (1.125cm);
\draw  (12.125,11.625) circle (1.125cm);
\draw  (14.75,9.625) circle (1.125cm);
\draw  (14.875,7.625) circle (1.125cm);
\draw  (16,8.5) circle (1.25cm);
\draw  (14.5,6) circle (1cm);
\draw  (13.25,5.25) circle (1.125cm);
\draw  (15.75,11.875) circle (1cm);
\draw  (17.125,12.5) circle (1cm);
\draw  (15.875,10.5) circle (1cm);
\draw  (14.375,12.375) circle (0.875cm);
\draw  (13.625,3.75) circle (1cm);
\draw [dashed] (6.25,12.625) -- (17.25,13.5);
\draw [dashed] (18.125,12.375) -- (17.25,8.125);
\draw [dashed] (17.25,8.125) -- (14.5,3.25);
\draw [dashed] (13.75,2.75) -- (7,3.25);
\draw [dashed] (6.5,3.875) -- (5.125,5.75);
\draw [dashed] (5.125,5.875) -- (4.25,9);
\draw [dashed] (4.25,9.125) -- (5.125,12.25);
\draw [dashed] (3.75,15.125) -- (4.875,14);
\draw [dashed] (5.125,15.125) -- (6.125,14);
\draw [dashed] (6.25,15) -- (7.25,14);
\draw [dashed] (7.5,15.125) -- (8.625,14);
\draw [dashed] (8.75,15.125) -- (9.875,14);
\draw [dashed] (10,15.125) -- (11.125,14);
\draw [dashed] (11.375,15) -- (12.375,14.125);
\draw [dashed] (12.5,15) -- (13.625,14.125);
\draw [dashed] (13.75,15.125) -- (14.875,14.125);
\draw [dashed] (15,15) -- (16.125,14);
\draw [dashed] (16.25,15) -- (17.375,14.125);
\draw [dashed] (17.5,14.875) -- (18.25,14.125);
\draw [dashed] (18.75,15) -- (19.75,14);
\draw [dashed] (20.125,15) -- (21,14);
\draw [dashed] (18.875,16.25) -- (19.875,15.25);
\draw [dashed] (2.625,16.375) -- (3.75,15.25);
\draw [dashed] (2.5,15.125) -- (3.75,14);
\draw [dashed] (1.375,15.125) -- (2.375,14);
\draw [dashed] (2.5,13.875) -- (3.625,12.75);
\draw [dashed] (2.625,12.625) -- (3.625,11.5);
\draw [dashed] (2.5,11.375) -- (3.625,10.25);
\draw [dashed] (2.5,10.125) -- (3.625,9);
\draw [dashed] (2.5,8.75) -- (3.75,7.625);
\draw [dashed] (2.5,7.625) -- (3.5,6.5);
\draw [dashed] (2.5,6.375) -- (3.625,5.125);
\draw [dashed] (2.5,5.125) -- (3.625,4);
\draw [dashed] (2.5,3.75) -- (3.5,2.75);
\draw [dashed] (2.5,2.625) -- (3.625,1.625);
\draw [dashed] (3.75,2.5) -- (4.75,1.5);
\draw [dashed] (5,2.5) -- (6.125,1.625);
\draw [dashed] (6.25,2.375) -- (7.25,1.625);
\draw [dashed] (7.375,2.375) -- (8.5,1.625);
\draw [dashed] (8.625,2.375) -- (10,1.5);
\draw [dashed] (9.875,2.375) -- (11.25,1.5);
\draw [dashed] (11.25,2.5) -- (12.25,1.625);
\draw [dashed] (12.5,2.5) -- (13.5,1.625);
\draw [dashed] (13.75,2.5) -- (14.875,1.625);
\draw [dashed] (15,2.5) -- (16.125,1.5);
\draw [dashed] (16.25,2.5) -- (17.375,1.625);
\draw [dashed] (17.5,2.5) -- (18.625,1.5);
\draw [dashed] (18.75,2.625) -- (20,1.625);
\draw [dashed] (20,2.375) -- (21,1.625);
\draw [dashed] (18.75,1.25) -- (19.625,0.5);
\draw [dashed] (2.625,1.25) -- (3.625,0.25);
\draw [dashed] (1.375,2.625) -- (2.375,1.625);
\draw [dashed] (18.75,3.75) -- (19.875,2.875);
\draw [dashed] (18.75,5.125) -- (19.75,4.125);
\draw [dashed] (18.75,6.25) -- (19.875,5.25);
\draw [dashed] (18.75,7.375) -- (19.75,6.625);
\draw [dashed] (18.75,8.625) -- (19.875,7.75);
\draw [dashed] (18.875,10.125) -- (19.75,9);
\draw [dashed] (18.75,11.25) -- (19.75,10.25);
\draw [dashed] (18.75,12.5) -- (19.75,11.5);
\draw [dashed] (18.75,13.625) -- (19.875,12.875);
\node [font=\fontsize{18.2pt}{23.7pt}\selectfont, inner xsep=0.080cm, inner ysep=0.085cm, rounded corners=0.020cm] at (9.25,9) {z .};
\end{circuitikz}}
\caption{A box surrounding the point $\bz$.}\label{fig:encircled}
\end{figure}
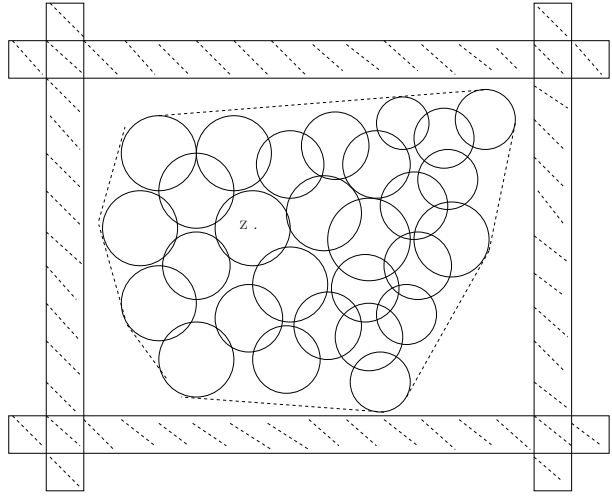

\begin{cor}
In the setup of Theorem~\ref{th3}, almost surely, every connected component of $\Loss_\infty$ is bounded.
\end{cor}
\begin{proof}
By the proof of Theorem~\ref{th3} and the second Borel--Cantelli lemma, almost surely there are infinitely many empty horizontal strips $S_h$ as $h\to+\infty$ and as $h\to-\infty$. Applying the same argument after rotating the plane by $90^\circ$, we also obtain infinitely many empty vertical strips to the left and to the right.

For each empty horizontal strip, its middle horizontal line is disjoint from $\Loss_\infty$, by the argument in the proof of Theorem~\ref{th3}. Similarly, each empty vertical strip contains a middle vertical line disjoint from $\Loss_\infty$. Hence, every point of the plane is enclosed in a finite rectangle whose boundary is disjoint from $\Loss_\infty$. Therefore, no connected component
of $\Loss_\infty$ can cross this rectangle, and every connected component of $\Loss_\infty$ is bounded. See Figure~\ref{fig:encircled}.
\end{proof}

We can also provide a weaker statement than Theorem~\ref{th3}, which covers slightly more general sets of intensities.

\begin{thm}\label{th3b}
Assume that for some constant $D>0$  all holes are bounded by $D$ in size, i.e.\ $\sup_{H\in\mathcal{H}} \diam(H)< D$, and the  Poisson process is such that
\begin{align}\label{eq:rfr}
\iint_{\R^2} \lambda^2(\bz)\,\dd{\bz}<\infty   .  
\end{align}
Then the membrane will not be destroyed a.s.
\end{thm}
\begin{rema}
    While Theorem~\ref{th3} covers the case with $f(r)=\frac{1}{r\log^\gamma r}$ for some $\gamma >1$, Theorem~\ref{th3b}  is applicable even if $\gamma>1/2$.
\end{rema}

\begin{proof}[Proof of Theorem~\ref{th3b}]
Since the property~\eqref{eq:rfr} is preserved under scaling, we can w.l.o.g.\ assume that $D=1$. Fix a pair $(i,j)\in\mathcal{I}$ where
$$
\mathcal{I}=\{(0,0),(0,1),(1,0),(1,1)\}
$$
is a set of four elements. Split the plane into disjoint squares of side length $4$:
$$
R_{a,b}^{i,j}=\{(x,y)\in\R^2:\ 4a-2<x-2i\le 4a+2,\ 4b-2<y-2j\le 4b+2\},\qquad a,b\in\Z.
$$
First, consider $(i,j)=(0,0)$. The number of hole centres in such a box is Poisson($\mu_{a,b}$) where 
\begin{align*}   
\mu_{a,b}&=\Lambda(R_{a,b}^{0,0})=\iint_{R_{a,b}^{0,0}}\lambda(\bz)\dd{\bz}
=\int_{4a-2}^{4a+2}\int_{4b-2}^{4b+2}\lambda(\bz)\dd{\bz}
\end{align*}
yielding
\begin{align}\label{Rmore2}
\P(R_{a,b}^{0,0}\text{ contains more than one hole centre})= \sum_{k=2}^\infty \frac{\mu_{a,b}^k e^{-\mu_{a,b}}}{k!}\le \mu_{a,b}^2.
\end{align}
At the same time,
\begin{align}\label{eq:mufinite}
\sum_{a,b\in\Z}\mu_{a,b}^2<\infty  .
\end{align}
Indeed, by the Cauchy--Schwarz inequality
\begin{align*}   
\mu_{a,b}^2=\left[\iint_{R_{a,b}^{0,0}}\lambda(\bz)\dd{\bz}\right]^2
\le 16 \iint_{R_{a,b}^{0,0}}\lambda^2(\bz)\dd{\bz}
\end{align*}
and now~\eqref{eq:mufinite} follows from~\eqref{eq:rfr}.

By the  Borel-Cantelli lemma, \eqref{Rmore2} and~\eqref{eq:mufinite} imply that almost surely there are only finitely many boxes $R_{a,b}^{0,0}$  with more than one hole center. The same argument applies to the other three shifted decompositions corresponding to \((i,j)\in\mathcal I\). Hence, almost surely, only finitely many boxes among all four decompositions contain more than one hole centre. Since two overlapping holes have centres at distance \(<2\), and any two points at distance \(<2\) lie in a common box \(R_{a,b}^{i,j}\) for at least one \((i,j)\in\mathcal I\) (see Figure~\ref{fig:chess}), it follows that there exists an a.s.\ finite random \(\nu\) such that no two holes with centres outside \(B_\nu(\mathbf 0)\) overlap.

\begin{figure}[ht]
    \centering
    \includegraphics[width=0.5\linewidth]{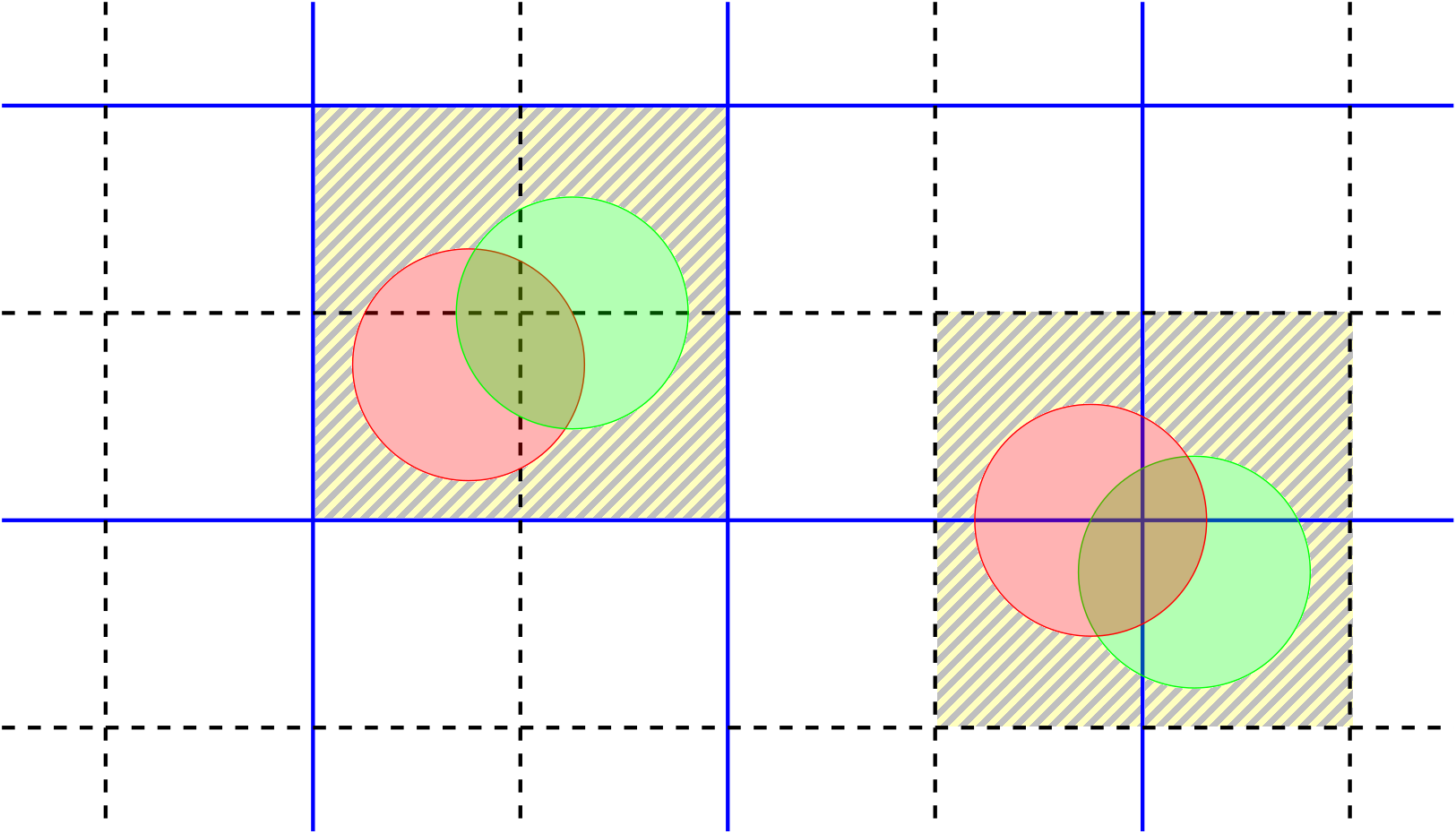}
    \caption{Two overlapping holes must have their centres in the same box $R_{a,b}^{i,j}$ for some $(i,j)\in\mathcal{I}$.}
    \label{fig:chess}
\end{figure}

Let $k>\nu/2$. Since all hole diameters are \(<1\), every hole whose centre lies in \(B_{2k}(\mathbf 0)\) is contained in \(B_{2k+1}(\mathbf 0)\). Similarly, every hole whose centre lies outside \(B_{2k+2}(\mathbf 0)\) is disjoint
from \(B_{2k+1}(\mathbf 0)\). Hence, if the annulus
\[
C_{2k}=B_{2k+2}(\mathbf 0)\setminus B_{2k}(\mathbf 0)
\]
contains no hole centres, then it separates the inside configuration from the outside configuration. In addition, since \(2k>\nu\),  all exterior holes remain isolated connected components. Consequently, applying \(\phi\) to the exterior part leaves it unchanged, while the component generated by
the holes inside \(B_{2k}(\mathbf 0)\) remains contained in
\(B_{2k+1}(\mathbf 0)\). Therefore, no propagation can cross the empty annulus, and in particular $\Loss_\infty\ne \mathbb R^2$. Now it only remains to show that such an empty annulus $C_{2k}$ exists a.s.

We have
$$
\P(C_{2k}\text{ has no hole centres})=e^{-\Lambda(C_{2k})}
$$
and by the Cauchy--Schwarz inequality again
$$
\Lambda(C_{2k})=\iint_{C_{2k}} \lambda(\bz)\dd{\bz}
\le \sqrt{\iint_{C_{2k}} \lambda^2(\bz)\dd{\bz} \iint_{C_{2k}} 1^2\dd{\bz}}
$$
yielding
$$
\frac{\Lambda(C_{2k})^2}
{\pi(8k+4)}\le \iint_{C_{2k}} \lambda^2(\bz)\dd{\bz}.
$$
Since the RHS is summable by~\eqref{eq:rfr} while $\sum_k (8k+4)^{-1}=\infty$, there are infinitely many~$k$ for which $\Lambda(C_{2k})^2<1$, and for each such $k$ we have $\P(C_{2k}\text{ has no hole centres})\ge e^{-1}$.
Because of the independence of these events, by the second Borel-Cantelli lemma there will be infinitely many $k$ for which this occurs, and thus we can always choose $k$ such that $2k>\nu$.

Consequently, after finitely many iterations the process stabilises with
finitely many bounded convex components completely inside of $B_{2k+1}(\mathbf{0})$ and only the original holes outside of it.
\end{proof}

\begin{rema}
Note that the fact that outside of a certain box $[-L,L]^2$ no two holes overlap does not by itself guarantee that the whole plane is not destroyed. Indeed, the construction in the proofs of Lemmas~\ref{lem2} and~\ref{lem3} shows just the opposite: once you have a ``good seed'' (e.g. inside the box $[-L,L]^2$), it is possible, even with quite sparsely located holes, to destroy the whole plane.
\end{rema}

\vspace{1cm}

We finish the article with a couple of additional remarks and an open problem.

\begin{rema}
Although all results in this paper are stated in two dimensions, they can in principle be extended to higher dimensions. Such extensions would require substantially more complicated geometric arguments without adding much insight into the main ideas, so we restrict attention to the two-dimensional setting.
\end{rema}

\begin{figure}[htbp]
\centering
\begin{tikzpicture}[scale=0.7]
\draw[->, very thick] (0,1) -- (6,4);
\draw[->,very thick] (2,0) -- (5,7);
\draw[dashed] (0,1) -- (2,0);
\draw[dashed] (2,0) -- (6,4);
\draw[dashed] (5,7) -- (6,4);
\draw[dashed] (0,1) -- (5,7);

\draw[->, very thick] (2,3) -- (-3,4);
\draw[->,very thick] (-2,3) -- (2.5,5);
\draw[dashed] (2,3) -- (2.5,5);
\draw[dashed] (2,3) -- (-2,3);
\draw[dashed] (-3,4) -- (2.5,5);
\draw[dashed] (-3,4) -- (-2,3);

\draw[->, loosely dashed] (5,0) -- (8,3) node[midway, below right] {$\bar\alpha$};  
\end{tikzpicture}
\caption{A convex hole (quadrilateral) obtained by taking the convex hull of two cuts}\label{fig:cuts}
\end{figure}

\begin{rema}
Another possible extension of our results would be the situation if we drop the requirement that the holes have non-empty interior. For example, imagine that there are ``cuts'' in the plane, where {\em a cut} is a straight line segment of random positive length $\xi$, which has the orientation $\alpha\in[0,\pi)$ (the angle with the horizontal axis). Assume that the centres of cuts are distributed according to a Poisson process with intensity~$\lambda$, their lengths and orientations are i.i.d., and that $\alpha$ is a non-degenerate random variable, i.e., not all cuts are parallel. Let $\bar\alpha$ be such that 
$\P(\alpha\le\bar\alpha)>0$ and $\P(\alpha>\bar\alpha)>0$.
For every two cuts which intersect such that the first (the second, respectively) cut has  the orientation $\le \bar\alpha$ ($>\bar\alpha$, respectively), define a hole as the quadrilateral obtained by taking the convex hull of these two cuts; see Figure~\ref{fig:cuts}. Thus, we will be able to use arguments similar to those in the proof of Corollary~\ref{COR}, albeit with a different $\lambda$. (An additional argument to handle the dependence is needed, but we believe it should not be very difficult.)
\end{rema}

\begin{prob}
Is there a positive function $g_\gamma(r)$ decreasing to zero as $r\to \infty$,  reasonably depending on a parameter $\gamma\in\R$ (for example of the type $g_\gamma(r)=\gamma f(r)$),  such that if the intensity of the PPP is given by $\l(\bz)=g_\gamma(\n{\bz})$, then $\P(\Loss_\infty=\R^2)=1$ if $\gamma>\gamma_{\mathrm{cr}}$ and $\P(\Loss_\infty=\R^2)<1$ if $\gamma<\gamma_{\mathrm{cr}}$ for some $\gamma_{\mathrm{cr}}$?
\end{prob}
A positive answer would establish a more interesting phase transition.


\end{document}